\documentclass[11pt]{amsart}
\usepackage{amsmath}
\usepackage{amsthm}
\usepackage{amssymb}
\usepackage{a4wide}
\usepackage{color}

\usepackage[pagewise,mathlines]{lineno}
\usepackage{hyperref}
\usepackage{amsfonts}

\usepackage{mathtools}
\mathtoolsset{showonlyrefs}

\newcommand*{\Z}{\mathbb Z}

\newcommand*{\R}{\mathbb R}

\newtheorem{lemma}{Lemma}[section]
\newtheorem{cor}{Corollary}[section]
\newtheorem{prop}{Proposition}[section]
\newtheorem{theorem}{Theorem}[section]

\newtheorem{remark}{Remark}[section]

\def\R{\mathbb{R}}

\def\eps{\varepsilon}

\begin{document}

\title[Nonlocal evolution equations in perforated domains]
{Nonlocal and nonlinear evolution \\ equations in perforated domains}

\author[M. C. Pereira and
Silvia Sastre-Gomez]{Marcone C. Pereira and
Silvia Sastre-Gomez}

\address{Marcone C. Pereira
\hfill\break\indent Dpto. de Matem{\'a}tica Aplicada, IME,
\hfill\break\indent
Universidade de S\~ao Paulo, \hfill\break\indent Rua do Mat\~ao 1010, 
S\~ao Paulo - SP, Brazil. } \email{{\tt marcone@ime.usp.br} \hfill\break\indent {\it
Web page: }{\tt www.ime.usp.br/$\sim$marcone}}

\address{Silvia Sastre-Gomez
\hfill\break\indent Dpto. de Matem{\'a}tica, CCEN,
\hfill\break\indent
Universidade Federal de Pernambuco, \hfill\break\indent Av. Prof. Moraes Rego, 1235, Recife - PE, Brazil. } 
\email{{\tt   silvia.sastre@dmat.ufpe.br} \hfill\break\indent 
%{\it Web page: }{\tt http://mate.dm.uba.ar/$\sim$jrossi/} 
}

\keywords{perforated domains, nonlocal equations, semilinear equations, Dirichlet problem.\\
\indent 2010 {\it Mathematics Subject Classification.} 45A05, 45M05, 49J40.}

\begin{abstract} 
In this work we analyze the behavior of the solutions to nonlocal evolution equations of the form 
$u_t(x,t) = \int J(x-y) u(y,t) \, dy - h_\eps(x) u(x,t) + f(x,u(x,t))$ 
with $x$ in a perturbed domain $\Omega^\eps \subset \Omega$ which is thought as a fixed set $\Omega$ from where we
remove a subset $A^\eps$ called the holes.
We choose an appropriated families of functions $h_\eps \in L^\infty$ in order to deal with both Neumann and Dirichlet conditions in the holes setting a Dirichlet condition outside $\Omega$.
Moreover, we take $J$ as a non-singular kernel and $f$ as a nonlocal nonlinearity.
Under the assumption that the characteristic functions of $\Omega^\eps$ have a weak limit, 
we study the limit of the solutions providing a nonlocal homogenized equation. 
\end{abstract}

\maketitle

\section{Introduction and main results}
\label{Sect.intro}
\setcounter{equation}{0}

Let $\Omega^\eps \subset \R^N$ be a family of open bounded sets  satisfying 
$\Omega^\eps \subset \Omega$
for some fixed open bounded domain $\Omega \subset \R^N$ and a positive parameter $\eps$. 
Denoting $\chi_\eps \in L^\infty(\R^N)$ by the characteristic function of $\Omega^\eps$,   
we assume that there exists a function $\mathcal{X} \in L^\infty(\R^N)$, strictly positive inside $\Omega$ 
such that 
$\chi_\eps \rightharpoonup \mathcal{X}$  weakly$^*$ in $L^\infty(\Omega)$.
More precisely, we suppose    
\begin{equation} \label{chic}
\int_{\Omega} \chi_\eps(x) \, \varphi(x) \, dx \to \int_{\Omega} \mathcal{X}(x) \, \varphi(x) \, dx, \quad \textrm{ as } \eps \to 0,
\end{equation}
for all $\varphi \in L^1(\Omega)$, and there exists a positive constant $c>0$ such that 
\begin{equation} \label{chib}
\mathcal{X}(x) \geq c > 0 \quad \textrm{ for all } x \in \Omega. 
\end{equation}

Notice that we also have $\mathcal{X}(x) \leq 1$ in $\Omega$ with $\mathcal{X}(x) \equiv 0$ in $\R^N \setminus \Omega$ since the family of characteristic functions $\chi_\eps$ satisfy the same conditions for all $\eps>0$.

Here, we see the family of open sets $\Omega^\eps$ as a family of perforated domains where the set 
$$
A^\eps = \Omega \setminus \Omega^\eps
$$
can be thought as the holes inside $\Omega$.
Our main goal is to analyze the asymptotic behavior of the solutions of a 
semilinear nonlocal evolution problem with nonlocal reaction and non-singular kernel in perforated domains as $\varepsilon \to 0$.

We consider problems of the form
 \begin{equation} \label{pintro}
%\left\{
\begin{gathered}
u_t(x,t) = \int_{\R^N \setminus A^\eps} J(x-y) u(y,t) \, dy - h_\eps(x) \, u(x,t)  + f(x, u(x,t)), \\
u(x,0) = u_0(x),
\end{gathered} 
%\right.
\end{equation}
 for $x \in \Omega^\eps$ and $t$ in bounded intervals 
taking nonlinearities  
 $f:\Omega^\eps \times L^1(\Omega^\eps) \mapsto \R$
 which are nonlocal reaction terms defined by 
 $$
 f(x,u) = (g \circ m_{\Omega^\eps})(x,u)
 $$
 under the following conditions:
 
${\bf (H_f)}$ We assume $g:\R \mapsto \R$ is a smooth function, globally Lipschitz, and 
$$m_{\Omega^\eps}: \Omega^\eps \times L^1(\Omega^\eps) \mapsto \R$$ 
is given by
$$
 m_{\Omega^\eps}(x,u) = \frac{1}{| B_\delta (x) \cap \Omega^\eps |} \int_{B_\delta (x) \cap \Omega^\eps } u(y) \, dy
 $$
where $B_\delta(x)$ is the ball of radius $\delta>0$ centered at $x \in \Omega^\eps$.

 Here, and along the whole paper, the function
$J$ is a smooth non-singular kernel satisfying  
$$
{\bf (H_J)} \qquad 
\begin{gathered}
J \in \mathcal{C}(\R^N,\R) \textrm{ is non-negative with } J(0)>0, \; J(-x) = J(x) \textrm{ for every $x \in \R^N$ and } \\
\int_{\R^N} J(x) \, dx = 1.
\end{gathered}
$$
 
 We consider both Dirichlet and Neumann nonlocal problems. For the Dirichlet case we impose $h_\eps(x) \equiv 1$ with 
$u$ vanishing in $\R^N \setminus \Omega^\eps$  
while in the Neumann case we consider 
$$
h_\eps(x) = \int_{\R^N \setminus A^\eps} J(x-y) \, dy, \quad x \in \Omega^\eps, 
$$
only assuming that $u$ vanishes in $\R^N\setminus \Omega$.
Note that for the former we have considered nonlocal Neumann boundary conditions
in the holes $A^\eps$ and a Dirichlet boundary condition in the exterior of the set $\Omega$.

It is not difficult to see that there are positive constants $\eps_0$ and $C_0$ such that 
 \begin{equation} \label{177}
|B_\delta(x) \cap \Omega^\epsilon | = \int_{B_\delta(x) \cap \Omega^\eps} \chi_\eps(y) \, dy  \geq C_0 
 \end{equation}
for all  $x \in \Omega$ and $0 < \eps < \eps_0$.
In particular, we have 
$\int_{B_\delta(x)} \mathcal{X}(y) \, dy  \geq C_0$ for all $x \in \Omega$.
 Indeed, inequality \eqref{177} follows from \eqref{chic} and \eqref{chib} since $\Omega \subset \R^N$ is a bounded open set.

See also that the map $m_{\Omega^\eps}$ transforms $(x, u)$ into $m_{\Omega^\eps}(x,u)$, the average of function $u$ in a set given by the intersection between the ball $B_\delta(x)$ with the perforated domain $\Omega^\eps$ setting the nonlocal nonlinear effect in our model.

According to \cite{Hut}, equation \eqref{pintro} can be seen as a continuous model for a single species in a finite $N$-dimensional habitat where the density of the population at position $x$ and time $t$ is given by a function $u(x,t)$.
Hostile surroundings are modeled by the Dirichlet conditions whereas the Neumann condition is the standard approach to modeling species in geographically isolated regions.  
The nonlocal effect under reaction terms is discussed for instance in \cite{Furter}. It is used to model situations where the total biomass plays a role and the model incorporates group defense or visual communications. 
See also \cite{Suarez11, Melian_sabina, react_chen}.

In fact, there exists a big interest in the study of nonlocal diffusion equations to model different problems from different areas. We still mention \cite{Allegretto, ElLibro, Fife, Freitas, Murray2002, Murray2003} and references therein where population dynamical processes and chemical reaction-diffusion models are treated. In \cite{gen_auger}, an economic model to fluctuation of stock market is presented.

 The paper is organized as follows: in Section \ref{mainresults}, we introduce the main results of the paper discussing the classic situation known as periodic perforated domains. In Section \ref{existence}, we study existence and uniqueness of the solutions to \eqref{pintro} obtaining uniform estimates on parameters $\eps$ and $\delta>0$. The proofs of our main results Theorem \ref{theoD} and Theorem \ref{theoN} are given in Section \ref{limit}. 
 
 In Section \ref{deltaS}, we discuss the asymptotic behavior of the limit equations, when parameter $\delta$ goes to zero obtaining a nonlocal problem with local nonlinearity. In this way, we give a scenario taking first $\varepsilon \to 0$, and next $\delta \to 0$. In a forthcoming paper, we will study the reversed limit. As noticed in \cite{nosotros2, nosotros}, a double limit commuting it is not expected. 
 
 Finally, we emphasize that the results obtained here are also in agreement with the previous works \cite{mjp,nosotros, mcp} 
where nonlocal linear equations have been considered. Therefore, this work is a natural continuation for nonlocal and nonlinear equations in perforated domains.

\section{Main results} \label{mainresults}

We have the following result for the Dirichlet problem:
\begin{theorem} \label{theoD}
Let $\{ u^\eps \}_{\eps>0}$ be the family of solutions given by problem \eqref{pintro} under conditions 
\begin{equation} \label{Dcond}
\begin{gathered}
u^\eps(x,t) \equiv 0 \quad \textrm{ for all } x \in \R^N \setminus \Omega^\eps \textrm{ and } t > 0 \\
u^\eps(x,0) = u_0(x)  \textrm{ in } L^2(\Omega) \\
\textrm{ and } \\
h_\eps(x) \equiv 1 \textrm{ in } \R^N. 
\end{gathered}
\end{equation}

Then, there exists $u^*: \R \times \R^N \mapsto \R$ with $u^*(x,t) \equiv 0$ in $\R^N \setminus \Omega$ and  
$u^* \in C^1([a,b], L^2(\R^N))$ for any closed interval $[a,b] \subset \R$,  such that, as $\eps \to 0$, 
$$
u^\eps \rightharpoonup u^* \textrm{ weakly$^*$ in } L^\infty([a,b] ; L^2(\Omega)).
$$

Furthermore, we have that limit function $u^*$ satisfies the following nonlocal equation in $\Omega$
\begin{eqnarray} 
u_t(x,t) & = & \mathcal{X}(x) \int_{\R^N} J(x-y) \, \left( u (y,t) - u(x,t) \right) dy  \\ 
& & \quad \quad \quad + \mathcal{X}(x) \, f_{\mathcal{X}}(x,u(x,t)) + (\mathcal{X}(x) - 1) \, u(x,t)   
\end{eqnarray}
with
\begin{equation} 
\begin{gathered}
u(x,t) \equiv 0 \quad \textrm{ in } x \in \R^N \setminus \Omega \textrm{ and } t>0 \\
u(x,0) = \mathcal{X}(x) \, u_0(x)
\end{gathered}
\end{equation}
where
\begin{equation} \label{fX}
f_{\mathcal{X}}(x,u) = (g \circ m_{\mathcal{X}})(x,u)
\end{equation}
and 
\begin{equation} \label{mX}
m_{\mathcal{X}}(x, u) = \frac{1}{\int_{B_\delta (x)} \mathcal{X}(y) \, dy} \int_{B_\delta (x)} u(y) \, dy.
\end{equation}
\end{theorem}

Concerning to Neumann conditions on the holes $A^\eps$ we have:
\begin{theorem} \label{theoN}
Let $\{ u^\eps \}_{\eps>0}$ be the family of solutions given by \eqref{pintro} with  
\begin{equation} \label{Ncond}
\begin{gathered}
u^\eps(x,t) \equiv 0 \quad \textrm{ for all } x \in \R^N \setminus \Omega \textrm{ and } t > 0 \\
u^\eps(x,0) = u_0(x)  \textrm{ in } L^2(\Omega) \\
\textrm{ and } \\
h_\eps(x) = \int_{\R^N \setminus A^\eps} J(x-y) \, dy  \quad  \textrm{ for } x \in \R^N.
\end{gathered}
\end{equation}

Then, there exists $u^*: \R \times \R^N \mapsto \R$ with $u^*(x,t) \equiv 0$ in $\R^N \setminus \Omega$ and  
$u^* \in \mathcal{C}^1([a,b], L^2(\R^N))$ for any closed interval $[a,b] \subset \R$,  such that, as $\eps \to 0$, 
$$
\tilde u^\eps \rightharpoonup u^* \textrm{ weakly$^*$ in } L^\infty([a,b] ; L^2(\Omega))
$$
where $\, \tilde \cdot \,$ denotes the extension by zero of functions defined on subsets of $\R^N$. 

Furthermore, we have the limit function $u^*$ satisfies the following nonlocal equation in $\Omega$
\begin{equation} \label{715}
\begin{gathered}
u_t(x,t) =  \displaystyle \mathcal{X}(x) \int_{\R^N} J(x-y)  \left( u (y,t) - u(x,t) \right) dy  \\
\qquad \qquad \qquad \qquad \qquad \qquad \displaystyle + \mathcal{X}(x) \, f_\mathcal{X}(x,u(x,t)) - \Lambda(x) \, u(x,t)  
\end{gathered}
\end{equation} 
with
\begin{equation} 
\begin{gathered}
u(x,t) \equiv 0 \quad \textrm{ for } x \in \R^N \setminus \Omega \textrm{ and } t>0 \\
u(0,x) = \mathcal{X}(x) \, u_0(x)
\end{gathered} 
\end{equation}
where the coefficient $\Lambda \in L^\infty(\R^N)$ is given by
$$
\Lambda(x) =  \int_{\R^N} J(x-y) \, ( 1- \chi_\Omega(y) + \mathcal{X}(y) ) \, dy - \mathcal{X}(x)
$$
and
$f_{\mathcal{X}} = g \circ m_{\mathcal{X}}$
with 
$m_{\mathcal{X}}$ defined by \eqref{mX}.

\end{theorem}

We point out the dependence of both limit equations on the term $\mathcal{X}$.
They establish the effect of the holes in the original equation \eqref{pintro}.
Indeed, a kind of friction or drag coefficient is obtained, as well as, a new reaction nonlinearity, both caused by the perforations.  
Also, if we rewrite the more involved term $\Lambda$ appearing in Theorem \ref{theoN} as 
$$
\Lambda(x) =  \int_{\R^N \setminus \Omega} J(x-y) \, dy + \int_{\R^N} J(x-y) 
\left( \mathcal{X}(y) - \mathcal{X}(x) \right) dy, \quad x \in \Omega, 
$$
we see that the kernel $J$ explicitly affects the limit equation for the Neumann problem.
As we can see, such dependence on the kernel $J$ does not occur in the Dirichlet problem where the coefficient only depends on the perturbations via $\mathcal{X}$.

Concerning to the extreme case $\mathcal{X}(x) \equiv 1$ in $\Omega$, we can argue as in \cite[Corollary 3.1]{mcp} to see that, 
if $\mathcal{X}(x) \equiv 1$ in $\Omega$, then the limit equation for both conditions is the nonlocal Dirichlet problem in $\Omega$, namely
\begin{equation} \label{DLP}
\begin{gathered}
u_t(t,x) = \int_{\R^N} J(x-y) (u(t,y) - u(t,x)) dy + \hat f(x,u(x,t)), \\
u(0,x) = u_0(x),
\end{gathered} 
\qquad x \in \Omega, \; t \in \R, 
\end{equation}
with 
$u(t,x) \equiv 0$ in $x \in \R^N \setminus \Omega$
and 
$\hat f:\Omega \times L^1(\R^N) \mapsto \R$
defined by 
 $$
\hat f(x,u) = (g \circ m_{B_\delta})(x,u)
 $$
where $m_{B_\delta}$ is the average of $u$ on the ball $B_\delta(x)$ 
\begin{equation} \label{m_delta}
m_{B_\delta}(x,u) = \frac{1}{|B_\delta(x)|} \int_{B_\delta(x)} u(y) \, dy.
\end{equation}
Hence, we can say that small holes do not make any effect on the limit process.  

Finally, we notice the degenerated case $\mathcal{X}(x) \equiv 0$ in $\R^N$ is not considered here, since we work under condition \eqref{chib}. It is a subject of a forthcoming paper.

See that the solutions of the limit problems given by Theorems \ref{theoD} and \ref{theoN} depend on the parameter $\delta>0$ which sets the average function $m_{\mathcal{X}}$ defined at \eqref{mX}.
Here, we also analyse the asymptotic behavior of these equations as $\delta \to 0$ under the additional conditions $J$, $\mathcal{X}$ and $u_0$ of class $\mathcal{C}^1$ and $g$ of class $\mathcal{C}^2$.
It is necessary to guarantee strong convergence in $L^2(\Omega)$ since we do not have regularizing effect for these nonlocal equations.

\begin{theorem} \label{theo_delta}
Let $\{ u^\delta \}_{\delta>0}$ be the family of solutions given by  
\begin{equation}\label{prob_delta_0}
\begin{gathered}
u^\delta_t(x,t) = \mathcal{X}(x) \int_{\R^N} J(x-y) \, u^\delta (y,t) \, dy  - h_0(x) \, u^\delta(x,t) + \mathcal{X}(x) \, f_\mathcal{X}(x, u^\delta(x,t)),   \\
u^\delta (x,t) \equiv 0 \quad \textrm{ for } x \in \R^N \setminus \Omega \\
u^\delta(0,x) = \mathcal{X}(x) \, u_0(x)
\end{gathered} 
\end{equation}
under conditions $J$, $\mathcal{X}$ and $u_0$ of class $\mathcal{C}^1$, $g$ of class $\mathcal{C}^2$, 
\begin{equation} \label{Dcond}
\begin{gathered}
u^\delta(x,t) \equiv 0 \quad \textrm{ for all } x \in \R^N \setminus \Omega \textrm{ and } t > 0 \\
u^\delta(x,0) =\mathcal{X}(x) u_0(x)  \textrm{ in } L^2(\Omega) \\
\textrm{ with } \\
h_0(x) \equiv 1 \textrm{ in } \R^N  \quad \textrm{ for the Dirichlet problem}\\
\textrm{ and } \\ h_0(x)=\int_{\R^N}J(x-y)(1-\chi_\Omega(y) + \mathcal{X}(y))\,dy  \quad \textrm{ for the Neumann problem. } 
\end{gathered}
\end{equation}

Then, there exists $\bar{u}: \R \times \R^N \mapsto \R$ with $\bar{u}(x,t) \equiv 0$ in $\R^N \setminus \Omega$ and  
$\bar{u} \in C^1([a,b], L^2(\R^N))$ for any closed interval $[a,b] \subset \R$,  such that, 
%as $\delta  \to 0$, 
$$
u^\delta \rightharpoonup \bar{u} \textrm{ weakly$^*$ in } L^\infty([a,b] ; L^2(\Omega)) ~~\textrm{ as }~ \delta \to 0.
$$

Furthermore, we have the limit function $\bar{u}$ satisfies the following nonlocal equation in $\Omega$
\begin{eqnarray} \label{1067}
{u}_t(x,t) & = & \mathcal{X}(x) \int_{\R^N} J(x-y) {u} (y,t) dy -h_0(x){u}(x,t) +\mathcal{X}(x) \, g\left(\frac{1}{\mathcal{X}(x)}{u}(x,t)\right)
\end{eqnarray}
with
${u}(x,t) \equiv 0$ in $x \in \R^N \setminus \Omega$ and 
${u}(x,0) = \mathcal{X}(x) \, u_0(x)$.
\end{theorem}

As expected, a local reaction term is obtained at the limit. The effect of the perforations can be seen, and a nontrivial term is captured. As in the previous results, under small perforations, that is, assuming $\mathcal{X}(x) \equiv 1$ in $\Omega$, a nonlocal Dirichlet equations in $\Omega$ is obtained for both Dirichlet and Neumann problems with a local reaction just set by function $g$. Under this additional condition Theorem \ref{theo_delta} implies the following limit equation 
\begin{equation} \label{DLP}
\begin{gathered}
u_t(t,x) = \int_{\R^N} J(x-y) (u(t,y) - u(t,x)) dy + g(u(x,t)), \\
u(0,x) = u_0(x),
\end{gathered} 
\qquad x \in \Omega, \; t \in \R, 
\end{equation}
with 
$u(t,x) \equiv 0$ in $x \in \R^N \setminus \Omega$.

\subsection*{Purely periodic perforations}

The study of solutions in periodic perforated domains has attracted much interest.
For local operators, from pioneering works to recent ones,
 we may mention \cite{ADMET, CJM, IP, CEW, DAPGR, GMM, N, RT, ESP} and references therein that are
concerned with elliptic and parabolic equations, nonlinear operators, as well as Stokes and Navier-Stokes equations from fluid mechanics. 
For instance, in the classical paper \cite{CM}, the authors analyze the Dirichlet problem for the Laplacian in a bounded domain 
from where a big number of periodic small balls are removed.  
They consider 
$$\Omega^\eps = \Omega \setminus \cup_i B_{r^\eps} (x_i)$$
where $B_{r^\eps} (x_i)$ is a ball centered 
in $x_i\in \Omega$ of the form $x_i \in 2 \eps \Z^N$ with radius $0 < r^\eps < \eps \leq 1$. 
Figure \ref{fig1} bellow ilustrates a periodic perforated domain $\Omega^\eps$. 

\begin{figure}[htp] 
\centering \scalebox{0.4}{\includegraphics{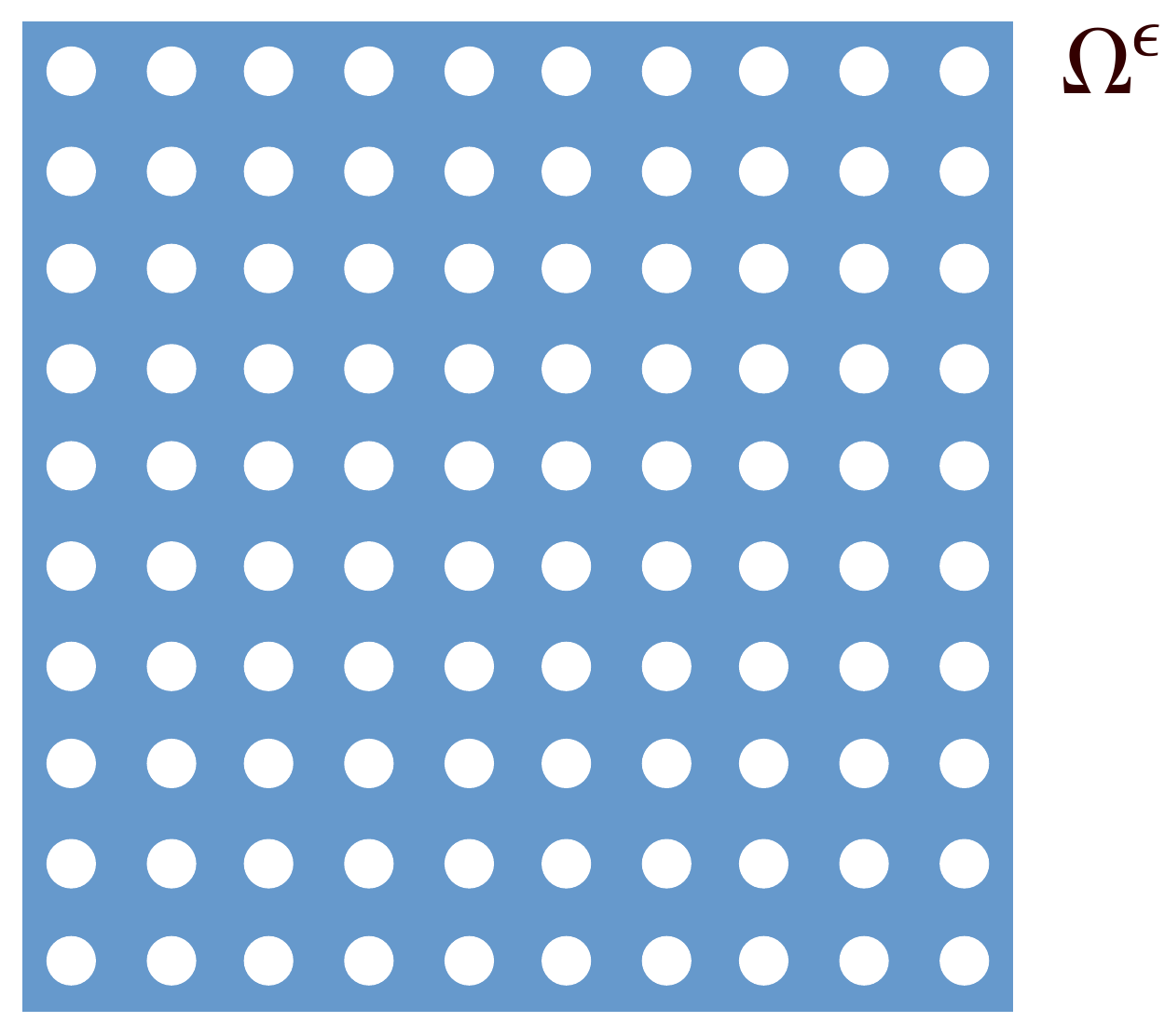}}
\caption{A periodic perforated domain $\Omega^\eps= (0,1)^2  \setminus \cup B_{r^\eps} (x_i)$.}
\label{fig1} 
\end{figure}

In this periodic case, it is known, see for instance \cite[Section 4]{nosotros}, that 
$$\chi_\eps \rightharpoonup \mathcal{X} = |Q\setminus B| /|Q| \quad \textrm{ as } \eps \to 0$$ 
where here, $Q$ denotes the unit cube and $B$ is a ball inside the cube. 
See that $\mathcal{X}$ is a positive constant. 
Thus, due to Theorems \ref{theoD} and \ref{theoN}, we have that the homogenized equations associated to \eqref{pintro} under Dirichlet and Neumann conditions are respectively:  
$$
\rho u_t(x,t) = \int_{\R^N} J(x-y) \left( u (y,t) - u(x,t) \right) dy + (1 - \rho)  u(x,t) + g\left( \rho \, m_{B_\delta}(x, u(\cdot,t) ) \right) 
$$
and  
\begin{eqnarray*}
\rho u_t(x,t) & = & \int_{\R^N} J(x-y) \left( u (y,t) - u(x,t) \right) dy  \\ 
& &  - \int_{\R^N \setminus \Omega} J(x-y) dy 
+ g\left( \rho \, m_{B_\delta}(x, u(\cdot,t) ) \right)
\end{eqnarray*}
for $x \in \Omega$, $t \in \R$ with $u(x,t) \equiv 0$ in $\R^N \setminus \Omega$ where 
$$\rho = {\mathcal{X}}^{-1} = |Q|/|Q \setminus B|$$ 
is a strictly positive constant larger than one since $|B|>0$ and $m_{B_\delta}$ is the average \eqref{m_delta}.

Notice that even in this classic and standard situation, the nonlinear term of \eqref{pintro} is perturbed in a non trivial way.
Indeed, we have $f_{\mathcal{X}}(\cdot,u) = g( \rho \, m_{B_\delta}(\cdot,u))$.
It is due to the fact that the integral operators considered here do not regularize, and hence solutions $u^\eps$ with initial conditions in $L^2$ are expected to be bounded in $L^2$ but nothing better. 

We can still assume appropriated conditions on $J$, $\mathcal{X}$ and $u_0$ to use Theorem \ref{theo_delta} and pass to the limit in the previous equations as $\delta \to 0$ obtaining a local reaction term, also depending on $\rho$ and given by 
$$
g( \rho \, u(\cdot,t) ).
$$

%For nonlocal evolution problems with smooth kernels we refer to \cite{ElLibro,BNS,cha,crew,ignat,MJ0,RS} and references therein. 
%%
%Furthermore, we still point out that the nonlocal problem considered here is a nonlocal analogous to the classical parabolic problem for the Laplacian with homogeneous Neumann boundary conditions in the holes, and Dirichlet on the exterior bound
%$$
%\left\{ 
%\begin{array}{ll}
%\displaystyle u_t = \Delta u + f,  \quad \textrm{ in } \Omega^\eps, \\
%\displaystyle u = 0, \quad \textrm{ on } \partial \Omega, \\
%\frac{\partial u }{\partial n} = 0, \quad \textrm{ on } \partial \Omega^\eps \setminus \partial \Omega.
%\end{array} \right.
%$$
%In \cite{crew} is shown that solutions to \eqref{1.1} converge, as a rescaling parameter that
%controls the size of the support of $J$ goes to zero, to the solution to this local problem. 
%On the other hand, it is worth noting that different behavior at the limit can be viewed in these analogous asymptotic problems with Neumann condition. See for instance \cite{MJ,MJ2}. In some sense, nonlocal approximations to local problems under singular perturbations have to be done in a very careful way.

%\medskip

\section{Existence and uniform boundedness} \label{existence}

In this section, we mainly prove existence and uniqueness of the solutions to problem 
\eqref{pintro} giving uniform bounds with respect to parameters $\eps$ and $\delta>0$. 
We also introduce a technical result concerning to the convergence of integral expressions under sequence of functions.

Let us consider here $B=\R^N\setminus A^\eps$ and $u(x)\equiv 0$ in $x\in \R^N\setminus \Omega$ for the Neumann problem and $B=\R^N$ and $u(x)\equiv 0$ in $x\in \R^N\setminus \Omega^\eps$  for the Dirichlet problem. Since we are assuming $J \in \mathcal{C}(\R^N,\R)$, we have that the operator  
\begin{equation} \label{Kop}
Ku(x)=\int_{B}J(x-y)u(y)dy.
\end{equation}
satisfies $K:L^2(\Omega)\to L^2(\Omega)$. On the other hand, it follows from \eqref{177} that 
$$a^{\eps}(x)=\frac{1}{| B_\delta (x) \cap \Omega^\eps |}$$ is uniformly bounded in $\eps$. Indeed, $a^{\eps}\in L^{\infty}(\Omega)$ and satisfies 
\begin{equation}\label{cota_a_x}
%	\frac{1}{\mu \left(B_{\delta}\right)}
	0\le a^{\eps}(x)\le \frac{1}{C_0},\quad\forall 
	x\!\in\!\Omega \textrm{ and } \eps \in (0,\eps_0).
\end{equation}

We are first interested in the Nemitcky 
operator associated to $f$, given by
\begin{equation} \label{mapF}
F:L^2(\Omega^\eps)\rightarrow L^2(\Omega^\eps)\quad \mbox{ with } \quad F(u)(x)=f(x,u)=g(m_{\Omega^\eps}(x,u)).
\end{equation}
To study the properties of $F$, we first see $M u(x)=m_{\Omega^\eps}(x,u)$. In the following lemma we state that  the Nemitcky operator $ M $ 
associated to $m_{\Omega^\eps}$ 
is continuous, globally Lipschitz and compact. For a proof,  see \cite{Zabreico_76}.

\begin{lemma}\label{f_react_dif_locally_lipschitz}
	Let $(\Omega, \mu, d)$ be a metric measure space with $\mu
	(\Omega)<\infty$, and set the operator 
	\begin{equation}\label{nemitck_op_M_definition}
		\displaystyle M (u)(x)=a_{\eps}(x)\int_{B_{\delta}(x)\cap \Omega^\eps}u(y)dy.
	\end{equation}
	
	Since the function $a_{\eps}\in L^{\infty}(\Omega)$ satisfies \eqref{cota_a_x}, we have that $M \in
	\mathcal{L}(L^1(\Omega), L^{\infty}(\Omega))$ 
	with $M : L^2(\Omega)\rightarrow L^{2}(\Omega)$ being a compact operator.
	\end{lemma}

\begin{remark} \label{Cconst}
In particular, if the nonlinear function $g:\R\to\R$ is globally Lipschitz, then the 
Nemitcky operator associated to $g$ and set by $G:L^2(\Omega)\rightarrow L^2(\Omega)$ is also
globally Lipschitz. Hence, since $M$ is a bounded operator by Lemma \ref{f_react_dif_locally_lipschitz}, we get $F=G\circ M:L^2(\Omega)\to L^2(\Omega)$ globally Lipschitz. 
Then, there exists a constant $C>0$, independent of $\eps>0$ such that 
$$
\| F(u) - F(v) \|_{L^2(\Omega)} \leq C \| u-v\|_{L^2(\Omega)}.
$$ 
\end{remark}

In the following Proposition, 
we prove the existence and uniqueness of the solutions to \eqref
{pintro} and we give a uniform bound of $u^\eps$ with respect to $\eps>0$.
\begin{prop}\label{exist}
The problem  \eqref{pintro}, under the assumptions
	\begin{equation*} 
\begin{gathered}
u^\eps(x,t) \equiv 0 \quad \textrm{ for all } x \in \R^N \setminus \Omega \textrm{ and } t > 0 \quad \textrm{ and } \\
h_\eps(x) = \int_{\R^N \setminus A^\eps} J(x-y) \, dy  \quad  \textrm{ for } x \in \R^N,
\end{gathered}
\end{equation*} 
	for the Neumann problem, or  
	\begin{equation*} 
\begin{gathered}
u^\eps(x,t) \equiv 0 \quad \textrm{ for all } x \in \R^N \setminus \Omega^\eps \textrm{ and } t > 0 \quad \textrm{ and } \\
h_\eps(x) \equiv 1 \textrm{ in } \R^N, 
\end{gathered}
\end{equation*}
	for the Dirichlet problem,
	has a unique global solution $u^\eps: \Omega \times \R \to \R$ with 
	$$u^\eps \in \mathcal{C}^1([a,b], L^2(\Omega^{\eps}))$$ 
	for every $u_0\in L^2(\Omega)$ and any bounded interval $[a,b] \subset \R$, with
	\begin{equation}\label{nonlinear_VCF_G_Lipschitz}
		u^{\eps}(\cdot,t)=e^{-h_{\eps}(x)t}u_0+\displaystyle\int_0^t e^{-h_{\eps}(x)(t-s)}[Ku+F (u^{\eps})](\cdot,s)
		\,ds \quad \forall t \in \R.
	\end{equation}
	
	Moreover,  there exists constants $\alpha$ and $D$, with $D>0$, independent of $\eps$ and 
	$t$, such that 
	\begin{equation}\label{unif_bound} 
	\|u^{\eps}(\cdot,t)\|_{L^2(\Omega^\eps)}\le e^{\alpha t}\left[\|u_0\|_{L^2(\Omega)} 
	+Dt\right].
	\end{equation}
\end{prop}
\begin{proof}
The existence and uniqueness result is proved using a fixed point argument with the variations of constants formula on the right of \eqref{nonlinear_VCF_G_Lipschitz} in 
$\mathcal{C}([-T, T], L^2(\Omega))$ for some $T>0$ independent of the initial data, and with a 
prolongation argument. Considering the formula on the right of \eqref{nonlinear_VCF_G_Lipschitz} 
as an operator defined 
from $L^{1}([-T, T], L^2(\Omega))$ into $\mathcal{C}([-T, T], L^2(\Omega))$, we have the uniqueness in both spaces and applying Theorem in \cite[p. 109]{Pazy} we obtain that $u$ is a strong solution of  (\ref{pintro}) in $\mathcal{C}^1([-T, T], L^2(\Omega))$.

To prove the uniform bound, we consider $B=\R^N\setminus A^\eps$ for the Neumann problem 
and $B=\R^N$ for the Dirichlet problem.
\[
\begin{array}{ll}
\displaystyle{d\over dt}&\displaystyle {1\over 2}\displaystyle \|u^\eps(\cdot,t)\|^2_{L^2(\Omega^\eps)} 
={d\over dt}{1\over 2}\int_{B} (u^\eps(x,t))^2dx=\int_{B} u^\eps(x,t) u_t^\eps(x,t)dx
\smallskip\\
&\displaystyle=\int_{B} u^\eps(x,t) \left[\int_{B} J(x-y) u^\eps(y,t) \, dy - h_\eps(x) \, u^\eps(x,t)  + f(x, u^\eps(x,t))\right]dx
\smallskip\\
&\displaystyle=-{1\over 2}\int_{B} \int_{B} J(x-y)( u^\eps(y,t)- u^\eps(x,t) )^2\,dy\,dx + \int_{B}u^\eps(x,t)f(x, u^\eps(x,t))\,dx
\end{array}
\]
Considering 
\[
\lambda_1^\eps=
\displaystyle \inf\limits_{u\in W} {\displaystyle{1\over 2}\int_{B} \int_{B} J(x-y)( u^\eps(y,t)- u^\eps(x,t) )^2\,dy\,dx\over \displaystyle \int_{B}\left(u^\eps(x,t)\right)^2\,dx} \]
where $W=\{u\in L^2(\R^N\!\setminus\! A^\eps): u(x)\equiv 0 \, \forall x\in \R^N\setminus \Omega\}$ for the Neumann problem, and $W=\{u\in L^2(\R^N): u(x)\equiv 0 \, \forall x\in \R^N\setminus \Omega^{\eps}\}$ for the Dirichlet problem. 
Thanks to Young's inequality and Remark \ref{Cconst}, since $g$ is globally Lipschitz, we have that 
\[
\begin{array}{ll}
\displaystyle{d\over dt} {1\over 2}\|u^\eps(\cdot,t)\|^2_{L^2(\Omega^\eps)} 
&\displaystyle\le  (\eta^2-\lambda^\eps_1)\int_{B} ( u^\eps(x,t) )^2\,dx + \eta^{-2}\|f(\cdot, u^\eps(\cdot,t))\|^2_{L^2(B)}
\smallskip\\&\displaystyle
\le (\eta^2-\lambda^\eps_1+\eta^{-2}C)\|u^\eps(\cdot,t)\|^2_{L^2(\Omega^\eps)} +|\Omega||g(0)| 
\end{array}
\]
for any $\eta>0$. Therefore integrating in $[0,t]$, we conclude 
\begin{equation}\label{bound}
\|u^{\eps}(\cdot,t)\|_{L^2(\Omega^\eps)}\le e^{2(\eta^2-\lambda^\eps_1+\eta^{-2}C) t}\left[\|u_0\|_{L^2(\Omega)} 
	+|\Omega||g(0)|t\right]
\end{equation}
finishing the proof.
\end{proof}

\begin{remark}
The term $\lambda_1^\eps$ is the first eigenvalue of 
\[
\int_B J(x-y)u^\eps(y)dy - h_\eps(x) \, u^\eps(x) - \lambda_1^{\eps}u^\eps(x)=0.
\]
From \cite{nosotros}, we know that the family $\lambda^\eps_1$ is lower bounded for both the Dirichlet and Neumann problems. For the Neumann problem, it is obtained under the additional condition: 
\begin{center}
There exists  finite family of sets $B_0, B_1, \dots, B_L\subset \R^N\setminus A^\eps$\\
such that $B_0=\R^N\setminus \Omega$, 
\[\R^N\setminus A^\eps\subset\bigcup_{i=0}^L B_i \quad \mbox{and}\quad \alpha_j={1\over 4}\min\limits_{x\in B_j}\int_{B_{J-1}}J(x-y)dy>0.\]
\end{center}
\end{remark}

Now let us study the uniform boundness with respect to $\delta$ of the solutions of the limit problems introduced by Theorems \ref{theoD} and \ref{theoN}. 
\begin{prop}\label{exist_delta}

	Let  $K\in\mathcal{L}(L^2(\Omega), L^2(\Omega))$ as in \eqref{Kop}, $g:\R\to \R$ globally Lipchitz and consider the map $F_{\mathcal{X}}: L^2(\Omega) \mapsto L^2(\Omega)$ set by 
	$F_{\mathcal{X}}(u)(x) =f_{\mathcal{X}}(x,u(x,t))=g(m_{\mathcal{X}}(x,u))$ where the functions $f_{\mathcal{X}}$ and $m_{\mathcal{X}}$ are given by \eqref{fX} and \eqref{mX}.
	
	Then $F_{\mathcal{X}}$ 
	is globally Lipschitz, and the problem  \eqref{prob_delta_0} has a unique global solution 
	$u^\delta: \Omega \times \R \to \R$ with 
	$$u^\delta \in \mathcal{C}^1([a,b], L^2(\Omega))$$ 
	for every $u_0\in L^2(\Omega)$, and any bounded interval $[a,b] \subset \R$, with
%	\begin{equation}\label{nonlinear_VCF_G_Lipschitz}
%		u^{\delta}(\cdot,t)=e^{-{h_{0}\over \mathcal{X}}(x)t}u_0+\displaystyle\int_0^t e^{-{h_{0}\over \mathcal{X}}(x)(t-s)}[Ku^\delta+F_X(u^{\delta})](\cdot,s)
%		\,ds \quad \forall t \in \R.
%	\end{equation}
	\begin{equation} \label{sol_u_delta}
\begin{array}{l}
\displaystyle u^\delta(x,t) = e^{-h_0(x)t} \, \mathcal{X}(x) \, u_0(x) + \int_0^t e^{-h_0(x)(t-s)} \, \mathcal{X}(x) \, f_\mathcal{X}(x, u^\delta(x,s)) \, ds \\[10pt] 
\qquad \qquad \qquad \qquad \displaystyle + \int_0^t e^{-h_0(x)(t-s)} \, \mathcal{X}(x) \int_{\R^N} J(x-y) \, u^\delta(y,s) \, dy\, ds. 
\end{array}
\end{equation}

	Moreover,  there exist constants $\alpha$ and $D$, with $D>0$, independent of $\delta$ and 
	$t$, such that 
	\begin{equation}\label{unif_bound} 
	\|u^{\delta}(\cdot,t)\|_{L^2(\Omega)}\le e^{\alpha t}\left[\|u_0\|_{L^2(\Omega)} 
	+Dt\right].
	\end{equation}
\end{prop}
\begin{proof}
Analogously to Proposition \ref{exist}, one can prove existence and uniqueness by fixed point arguments with variations of constants formula on the right of \eqref{sol_u_delta} in $\mathcal{C}([-T, T], L^2(\Omega))$ for any $T>0$ with a prolongation argument. 
%Considering the formula on the right of \eqref{nonlinear_VCF_G_Lipschitz} 
%as an operator defined 
%from $L^{1}([-T, T], L^2(\Omega))$ into $\mathcal{C}([-T, T], L^2(\Omega))$, we have the uniqueness in both spaces and applying Theorem in \cite[p. 109]{Pazy} we obtain that $u$ is a strong solution of  (\ref{pintro}) in $\mathcal{C}^1([-T, T], L^2(\Omega))$.

To prove the boundness, we consider the problem re-scaling $t$ with $t = \mathcal{X}^{-1}(x) \, \tau$ and setting 
$$
w^\delta(x,\tau) = u^\delta(x, {\mathcal{X}(x)}^{-1} \tau).
$$
We have that $w^\delta$ satisfies the equation
\begin{equation} \label{eq_rescaled}
\begin{gathered}
\displaystyle w^\delta_t(x,t)=\int_{\R^N}J(x-y)\,w^\delta(y,t)dy-{h_0(x)\over \mathcal{X}(x)}w^\delta(x,t)+f_{\mathcal{X}}(x,w^\delta(x,t))  \quad x\in \Omega
\smallskip\\
w^\delta(x,t)\equiv 0 \quad  x\in\R^N\setminus\Omega
\end{gathered}
\end{equation}
 with ${h_0(\cdot)\over \mathcal{X}(\cdot)} \in L^\infty(\Omega)$. 
Let us prove the uniform bound for $w^\delta$.
\[
\begin{array}{ll}
\displaystyle{d\over dt}\displaystyle {1\over 2}\displaystyle \|w^\delta(\cdot,t)\|^2_{L^2(\Omega)} 
&\displaystyle ={d\over dt}{1\over 2}\int_{\Omega} (w^\delta(x,t))^2dx=\int_{\Omega} w^\delta(x,t) w_t^\delta(x,t)dx
\smallskip\\
&\displaystyle=\int_{\Omega} w^\delta(x,t) \left[\int_{\Omega} J(x-y) w^\delta(y,t) \, dy - {h_0(x)\over \mathcal{X}(x)} \, w^\delta(x,t)  + f_{\mathcal{X}}(x, w^\delta(x,t))\right]dx
\smallskip\\
%\displaystyle=\int_{\Omega} w^\delta(x,t) \left[\int_{\Omega} J(x-y) w^\delta(y,t) \, dy-h_0(x)w^{\delta}(x,t)+  h_0(x)w^{\delta}(x,t) - {h_0(x)\over \mathcal{X}(x)} \, w^\delta (x,t)  + f(x, w^\delta (x,t))\right]dx
%\smallskip\\
&\displaystyle=-{1\over 2}\!\int_{\Omega} \!\int_{\Omega}\!\!\!\! J(x-y)( w^\delta(y,t)\!-\! w^\delta(x,t) )^2\,dy\,dx\! \\
&\displaystyle+ \int_{\Omega}\!\!\!w^\delta(x,t)\!\left(\!f_{\mathcal{X}}(x, w^\delta(x,t))\!+\!\left(\!\tilde{h}(x)\!-\! {h_0(x)\over \mathcal{X}(x)}\right)w^\delta(x,t) \right)dx
\end{array}
\]
where $\tilde{h}(\cdot)=\int_{\Omega}J(\cdot-y)dy\in L^\infty(\Omega)$. Take  
\begin{equation} \label{lam1}
\lambda_1=
\displaystyle \inf\limits_{w\in W} {\displaystyle{1\over 2}\int_{\Omega} \int_{\Omega} J(x-y)( w(y,t)- w(x,t) )^2\,dy\,dx\over \displaystyle \int_{\Omega}(w(x,t) )^2\,dx}
\end{equation}
where $W=\{w\in L^2(\Omega): w(x)\equiv 0 \, \forall x\in \R^N\setminus \Omega\}$. 
%for the Neumann problem, and $W=\{u\in L^2(\Omega): u(x)\equiv 0 \, \forall x\in \R^N\setminus \Omega^{\eps}\}$ for the Dirichlet problem. 
Thanks to Young's inequality and since $g$ is globally Lipschitz, we have that 
\[
\begin{array}{l}
\displaystyle{d\over dt} {1\over 2}\|w^\delta(\cdot,t)\|^2_{L^2(\Omega)} \!\!\!\!

\\ \displaystyle\!\le\!  (\eta^2\!\!-\!\!\lambda_1)\!\int_{\Omega}\!\! ( w^\eps(x,t) )^2dx \!+\! \eta^{-2}\!\left\|f_{\mathcal{X}}(\cdot, w^\delta(\cdot,t))\!+\!\left(\!\tilde{h}(\cdot)\!-\! {h_0(\cdot)\over \mathcal{X}(\cdot)}\right)\!w^\delta(\cdot,t)\right\|^2_{L^2(\Omega)}
\smallskip\\ 
\displaystyle
\!\le\! (\eta^2-\lambda_1+\eta^{-2}C)\|w^\delta(\cdot,t)\|^2_{L^2(\Omega)} +|\Omega||g(0)|. 
\end{array}
\]
Therefore integrating in $[0,t]$, we conclude 
\begin{equation}\label{bound}
\|w^{\delta}(\cdot,t)\|_{L^2(\Omega^\eps)}\le e^{2(\eta^2-\lambda^\delta_1+\eta^{-2}C) t}\left[\|\mathcal{X}u_0\|_{L^2(\Omega)} 
	+|\Omega||g(0)|t\right]
\end{equation}
finishing the proof.
\end{proof}

\begin{remark}
The term $\lambda_1$ introduced in \eqref{lam1} is known as the first eigenvalue of 
\[
\int_\Omega J(x-y)( u(y) - u(x)) \, dy - \lambda_1u(x)=0.
\]
See \cite{ElLibro} for more details. 
\end{remark}

\begin{remark} \label{reE}
Notice that the results of existence and uniqueness stated in Propositions \ref{exist} and \ref{exist_delta} are also valid for the problem 
\begin{equation} \label{reESolution}
\begin{gathered}
\displaystyle\omega_t(x,t)=\int_{\R^N}J(x-y)(\omega(y,t)-\omega(x,t))dy-h_0(x)\omega(x,t)+f_{\mathcal{X}}(x,\omega(x,t))  \quad x\in \Omega
\smallskip\\
\omega(x,t)\equiv 0 \quad  x\in\R^N\setminus\Omega
\end{gathered}
\end{equation}
for any $h_0 \in L^{\infty}(\Omega)$ and $f_{\mathcal{X}}$ given as in Theorems \ref{theoD} and \ref{theoN}.
\end{remark}

\begin{cor} \label{cordelta}
Let us assume under hypotheses of Proposition \ref{exist_delta} the additional conditions: $J$, $\mathcal{X}$ and $u_0$ of class $\mathcal{C}^1$ and $g$ of class $\mathcal{C}^2$ in $\R^N$.
Then the family of solutions $u^\delta(\cdot, t)$ of the problem \eqref{prob_delta_0} given by \eqref{sol_u_delta} belongs to $L^2([a,b], H^1(\Omega))$ and satisfies
	\begin{equation} \label{estH1}
	\|u^{\delta}(\cdot,t)\|_{H^1(\Omega)}\le C e^{\alpha t} \|u_0\|_{H^1(\Omega)}  \quad \forall t \in [a,b],
	\end{equation}
	for constants $\alpha$ and $D$, with $D>0$, independent of $\delta$ and $t$.
\end{cor}
\begin{proof}
First we notice that under the additional conditions the function $u^\delta(\cdot,t)$ defined by \eqref{sol_u_delta} belongs to $H^1(\Omega)$. In fact, since $J$ is a function of class $\mathcal{C}^1$, we have that $h_0$ is also $\mathcal{C}^1$.  
Consequently, we get 
$e^{-h_0t} \, \mathcal{X} \, u_0$ and  $\int_0^t e^{-h_0(t-s)} \, \mathcal{X} \int_{\R^N} J(\cdot-y) \, u^\delta(y,s) \, dy \,ds$ in $H^1(\Omega)$ for all $t$ since $\mathcal{X}$ and $u_0$ are also of class $\mathcal{C}^1$. 
It remains to show that 
$$
\Phi_t(x) = \int_0^t e^{-h_0(x)(t-s)} \, \mathcal{X}(x) \, f_\mathcal{X}(x, u^\delta(x,s)) \, ds, \quad x \in \Omega, 
$$
belongs to $H^1(\Omega)$ to conclude that $u^\delta(\cdot,t) \in H^1(\Omega)$ for all $t \in \R$. Indeed, we have that $\Phi_t \in H^1(\Omega)$ if and only if $f_\mathcal{X}(\cdot, u^\delta(\cdot,t)) \in H^1(\Omega)$, which is in $H^1(\Omega)$ if and only if, $m_\mathcal{X}(\cdot, u^\delta(\cdot,t)) \in H^1(\Omega)$ since $g$ is a Lipschitz function. 
Then, let us see that $m_\mathcal{X}(\cdot, u^\delta(\cdot,t)) \in H^1(\Omega)$. But, we notice that, this is a direct consequence from \cite[Lebesgue-Radon-Nikodym Theorem]{Rudin}. Since  
$$
m_{\mathcal{X}}(x, u^\delta) = \frac{1}{\int_{B_\delta (x)} \mathcal{X}(y) \, dy} \int_{B_\delta (x)} u^\delta(y) \, dy,
$$
with  $u^\delta \in L^2(\R^N)$ and $\mathcal{X} \in L^\infty(\R^N)$ satisfying $\mathcal{X}(x) \geq c>0$ for all $x \in \Omega$, we have that $m_\mathcal{X}(\cdot, u^\delta(\cdot,t))$ is an absolute continuous function, and then, it belongs to $H^1(\Omega)$.

Next, let us see which is the expression of the partial derivative $u^\delta_{x_i}$ taking into account that $u^\delta$ 
is given by \eqref{sol_u_delta}. By \eqref{derivative} in the appendix, and performing the appropriate computations, we get  
\begin{equation} \label{der_u_delta}
\begin{array}{l}
\displaystyle u_{x_i}^\delta(x,t) \displaystyle = -{\partial h_0\over \partial x_i}(x) t e^{-h_0(x)t} \, \mathcal{X}(x) \, u_0(x)
+ {\partial \mathcal{X}\over \partial x_i}(x)e^{-h_0(x)t} \, u_0(x)
+ {\partial u_0\over \partial x_i}(x)e^{-h_0(x)t} \, \mathcal{X}(x)
\smallskip\\

\displaystyle + \int_0^t \!\!\!\left(\!-{\partial h_0\over \partial x_i}(x)(t-s)\mathcal{X}(x)+{\partial \mathcal{X}\over \partial x_i}(x)\right)\!e^{-h_0(x)(t-s)} \left(\!f_\mathcal{X}(x, u^\delta(x,s)) +\!\!\!\int_{\R^N} \!\!\!\!\!\!J(x-y) u^\delta(y,s) dy\right) ds 
\smallskip\\
 \displaystyle + \int_0^t e^{-h_0(x)(t-s)}\mathcal{X}(x)\int_{\R^N} {\partial J\over\partial x_i}(x-y) \, u^\delta(y,s) dy\, ds
\smallskip\\
\displaystyle + \int_0^t e^{-h_0(x)(t-s)} \mathcal{X}(x) g'(m_{\mathcal{X}}(x, u^\delta(x,s)) 
\left(
{-\int_{B_{\delta}(x)}{\partial\mathcal{X}\over \partial x_i}(y)dy\over \left(\int_{B_{\delta}(x)}\mathcal{X}(y)dy\right)^2}\int_{B_{\delta(x)}}\!\!\!\!\!\!u^\delta(y,s)dy\right)\,ds
\smallskip\\
\displaystyle+ \int_0^t e^{-h_0(x)(t-s)} \mathcal{X}(x) g'(m_{\mathcal{X}}(x, u^\delta(x,s)) 
 \left({1\over \int_{B_{\delta}(x)}\mathcal{X}(y)dy}\int_{B_{\delta(x)}}\!\!\!{\partial u^\delta\over \partial x_i}(y,s)dy\right)\, ds. 
\end{array}
\end{equation}
Now, considering $L^2(\Omega)$ norm on the previous expression, since $h_0, \,\mathcal{X}, \,{1\over \mathcal{X}}$ and ${h_0\over\mathcal{X}}\in W^{1,\infty}(\Omega)$, $\mathcal{X}$ and $u$ satisfy \eqref{chib} and \eqref{unif_bound} respectively, with $\Omega \subset \R^N$ bounded, we obtain from  
H{\"o}lder inequality and Hardy-Littlewood maximal inequality that 
\begin{equation} \label{norm_der_u_delta_1}
\begin{array}{l}
\displaystyle \|u_{x_i}^\delta(\cdot, t)\|_{L^2(\Omega)} \displaystyle \le  \left(\left\|\partial h_0\over \partial x_i\right\|_{L^{\infty}(\Omega)} \|\mathcal{X}\|_{L^{\infty}(\Omega)} t 
+ \left\|{\partial \mathcal{X}\over \partial x_i}\right\|_{L^{\infty}(\Omega)} 
+\|\mathcal{X}\|_{L^{\infty}(\Omega)}\right) \| u_0\|_{H^1(\Omega)} e^{\tilde{\alpha} t} 

\smallskip\\
\displaystyle + \int_0^t \!\!\!\left(\left\|\partial h_0\over \partial x_i\right\|_{L^{\infty}(\Omega)} \|\mathcal{X}\|_{L^{\infty}(\Omega)} (t-s)+ \left\|{\partial \mathcal{X}\over \partial x_i}\right\|_{L^{\infty}(\Omega)}\right) e^{\tilde{\alpha}(t-s)}\!\!\left( (L_g+1)e^{\alpha s}(\|u_0\|_{L^2(\Omega)}+Ds)\right)ds
 
\smallskip\\
 \displaystyle + \int_0^t e^{\tilde{\alpha}(t-s)}\|\mathcal{X}\|_{L^{\infty}(\Omega)}
 \left\|{\partial J\over\partial x_i}\right\|_{L^2(\Omega)} e^{\alpha s}\left(\|u_0\|_{L^2(\Omega)}+Ds\right)ds
 
 \smallskip\\
\displaystyle
+ \int_0^t e^{\tilde{\alpha}(t-s)}\|\mathcal{X}\|_{L^{\infty}(\Omega)}\left(L_{g'}e^{\alpha s}(\|u_0\|_{L^2}+Ds)+|g'(0)||\Omega|^{1/2}\right)\!\!\left(\!\! {C_{2,N} \over c^2}\left\|{\partial \mathcal{X}\over \partial x_i}\right\|_{L^{\infty}(\Omega)}\!\!\!\!\!\!\!\!\!\!\!\!(\|u^{\delta}(\cdot, s)\|_{L^2(\Omega)}\!\!\right)ds
 
\smallskip\\
\displaystyle + \int_0^t e^{\tilde{\alpha}(t-s)}\|\mathcal{X}\|_{L^{\infty}(\Omega)}\left(L_{g'}e^{\alpha s}(\|u_0\|_{L^2}+Ds)+|g'(0)||\Omega|^{1/2}\right)\left( {C_{2,N} \over c} \|u^\delta_{x_i}(\cdot, s)\|_{L^2(\Omega)}\right)ds,
\end{array}
\end{equation}
since $\|e^{h_0(x)t}\|_{L^\infty}\le e^{\tilde{\alpha} t}$ for some constant $\tilde \alpha$, $|\int_{\Omega}J(x-y)u(y)dy|\le \|J\|_{L^{p'}}\|u\|_{L^p}$ in $\R^N$,  and 
$$
\begin{gathered}
\|g'(m_{\mathcal{X}}(\cdot, u^\delta(\cdot,s))\|_{L^2} \le L_{g'}\|u^\delta(\cdot,s)\|_{L^2}+|g'(0)||\Omega|^{1/2} \\
\le L_{g'}e^{\alpha s}(\|u_0\|_{L^2}+Ds)+|g'(0)||\Omega|^{1/2}.
\end{gathered}
$$

Then, for any $T>0$, we have  
\begin{equation} \label{norm_der_u_delta_3}
\begin{gathered}
 \|u_{x_i}^\delta(\cdot, t)\|_{L^2(\Omega)}  \le 
 C_1(h_0,\mathcal{X},J,g,T,D,\Omega)\|u_0\|_{H^1(\Omega)} \\
 \qquad \qquad \qquad \qquad \qquad 
 + C_2(h_0,\mathcal{X},u_0,g,D,\Omega,N,2,c)\int_{0}^t \|u^{\delta}_{x_i}(\cdot, s)\|_{L^2(\Omega)}ds, \quad \forall t \in [0,T].
\end{gathered}
\end{equation}
Thanks to Gr{\"o}nwall's inequality, we obtain
\begin{equation} \label{norm_der_u_delta_3}
 \|u_{x_i}^\delta(\cdot, t)\|_{L^2(\Omega)}  \le \
 C_1(h_0,\mathcal{X},J,g,b,D,\Omega)\|u_0\|_{H^1(\Omega)}e^{C_2(h_0,\mathcal{X},u_0,g,D,\Omega,N,2,c)t}.
\end{equation}
Thus, we can conclude the proof.
\end{proof}

Finally, we would like to present a basic fact that will be need in the sequel. 
The proof may be seen in \cite{mcp}. 
\begin{prop} \label{propL}
Let $\varphi^\eps$ be a sequence in $L^p(\R^N)$ with $1<p\leq \infty$ which vanishes in $\R^N \setminus \Omega$. 
Suppose that, as $\eps \to 0$,   
\begin{equation} \label{eq334}
\begin{gathered}
\varphi^\eps \rightharpoonup \varphi \quad \textrm{ weakly in } L^p(\Omega), \textrm{ as } 1<p<\infty, \\
\textrm{ or } \quad \varphi^\eps \rightharpoonup \varphi \quad \textrm{ weakly$^*$ in } L^\infty(\Omega), \textrm{ as } p=\infty, \\
\end{gathered}
\end{equation}
for some $\varphi$ in $L^p(\R^N)$ also satisfying $\varphi(x)\equiv 0$ in $\R^N \setminus \Omega$.
Then, if $J$ holds hypothesis ${\bf (H_J)}$,
$$
\Phi^\eps(x) = \int_{\R^N} J(x-y) \varphi^\eps(y) \, dy 
 \to  
\Phi_0(x) = \int_{\R^N} J(x-y) \varphi(y) \, dy, \quad \textrm{ as } \eps \to 0,
$$
strongly in $L^p(\mathcal{O})$ for any compact set $\mathcal{O} \subset \R^N$.
\end{prop}
%\begin{proof} 
%First, we notice that hypothesis \eqref{eq334} implies that $\Phi^\eps(x) \to \Phi_0(x)$ for all $x \in \R^N$.
%Also, if $p^{-1} + q^{-1} =1$, we have
%$| \Phi^\eps(x)| \leq \| \varphi^\eps\|_{L^p(\Omega)} \| J(x-\cdot) \|_{L^q(\Omega)}$.
%Thus, the result is a direct consequence of Arzel\`a-Ascoli Theorem since the restriction of $J$ to any compact set in $\R^N$ is a uniformly continuous function.
%\end{proof}

\section{The limit equations} \label{limit}

In this section we prove Theorems \ref{theoD} and \ref{theoN}.

First we notice that the existence of the family of solutions $u^\eps$ of \eqref{pintro} under conditions \eqref{Dcond} and \eqref{Ncond} are guaranteed by Proposition \ref{exist}. 
Also, there exists a positive constant $C$, 
independent of $\eps>0$, such that, for any bounded interval $[a,b] \subset \R$, 
\begin{equation} \label{Kest0}
\sup_{t \in [a,b]} \| u^\eps(\cdot,t) \|_{L^2(\Omega^\eps)} \leq C.
\end{equation}
Hence, if $\tilde \cdot$ denotes the extension by zero to the whole space $\R^N$, we also get that 
\begin{equation} \label{Kest}
\sup_{t \in [a,b]} \| \tilde u^\eps(\cdot,t) \|_{L^2(\Omega)} \leq C,
\end{equation}
and then, $\tilde u^\eps$ sets a uniformly bounded family in $L^\infty \left( [a,b] ; L^2(\Omega)  \right)$.

Moreover, if $\chi_\eps$ is the characteristic function of $\Omega^\eps$, then
$ \tilde u^\eps (x) = \chi_\eps(x) \, u^\eps(x)$.
Notice that in the Dirichlet case we have $\tilde u^\eps = u^\eps$ by condition \eqref{Dcond}. We keep the notation just to simplify the proof.

Also, since $L^1 \left( [a,b] ; L^2(\Omega)  \right)$ is separable, 
we can extract a subsequence, still denoted by $\tilde u^\eps$, such that
\begin{equation} \label{weak}
\tilde u^\eps \rightharpoonup u^* \textrm{ weakly$^*$ in } L^\infty([a,b] ; L^2(\Omega)),
\end{equation}
for some $u^* \in L^\infty \left( [a,b] ; L^2(\Omega)  \right)$. Notice that $u^*(x,t) \equiv 0$ in $\R^N \setminus \Omega$.

\begin{proof}[Proof of Theorems \ref{theoD} and \ref{theoN}.]
From now on, we assume, without loss of generality, that $[a,b] = [0,T]$ for some $T>0$.
We pass to the limit in the variational formulation of the expression \eqref{nonlinear_VCF_G_Lipschitz}.
That is, for any $\varphi \in L^2(\Omega)$, we pass to the limit in the following form
\begin{equation} \label{550}
\begin{array}{ll}
\displaystyle \int_{\Omega} \varphi(x) \, \tilde u^\eps(x,t) \, dx & = 
\displaystyle
\int_{\Omega} \varphi(x) \, e^{-h_\eps(x) t} \, \chi_\eps(x) \, u_0(x) \, dx 
\\[10pt]
& \quad \displaystyle + \int_{\Omega} \varphi(x) \, \chi_\eps(x) \int_{0}^t e^{-h_\eps(x) (t-s)} \, f(x,\tilde u^\eps(x,s)) \, ds  dx  \\[10pt] 
&  \quad \displaystyle + 
\int_{\Omega} \varphi(x) \, \chi_\eps(x) \int_{0}^t e^{-h_\eps(x) (t-s)} \int_{\R^N} J(x-y) \, \tilde u^\eps(y,s) \, dy ds dx  \\[10pt]
&  \displaystyle  =  I_1^\eps + I_2^\eps + I_3^\eps.
\end{array}
\end{equation}

Since condition \eqref{Ncond} is much more involved, we will just present the proof under this assumption. 
The Dirichlet problem is simpler. 
First, we evaluate $I_1^\eps$. Due to \eqref{Ncond}, we have for any $x \in \R^N$ that 
$$
h_\eps(x) = \int_{\R^N \setminus A^\eps} J(x-y) \, dy
= \int_{\R^N} J(x-y) \left( 1 - \chi_\Omega(y) + \chi_\eps(y) \right) dy 
$$
where $\chi_\Omega$ is the characteristic functions of the open bounded set $\Omega$. 
Then, from assumption \eqref{chic}, it follows from Proposition \ref{propL} that
$$
h_\eps \to h_0 \quad \textrm{ strongly in } L^\infty(\Omega)
$$
where $h_0 \in L^\infty(\R^N)$ is given by 
\begin{equation} \label{a0}
h_0(x) = \int_{\R^N} J(x-y) \left( 1- \chi_\Omega(y) + \mathcal{X}(y) \right) dy.
\end{equation}
Consequently, we obtain that
\begin{equation} \label{eq537}
e^{h_\eps(x) t}  \to e^{h_0(x) t} \quad \textrm{ uniformly in } (x,t) \in [0,T] \times \Omega,
\end{equation}
and then, $I_1^\eps = \int_{\Omega} \varphi(x) \, e^{-h_\eps(x) t} \,  \chi_\eps(x) \, u_0(x) \, dx$ satisfies 
\begin{equation} \label{I1}
I^\eps_1 \to \int_\Omega \varphi(x) \, e^{-h_0(x) t} \, \mathcal{X}(x) \, u_0(x)\, dx.
\end{equation}

Notice that for the Dirichlet condition \eqref{Dcond}, we have $h_\eps(x) \equiv 1$, and then, we get
\begin{equation} \label{I1D}
I^\eps_1 \to \int_\Omega \varphi(x) \, e^{- t} \, \mathcal{X}(x) \, u_0(x)\, dx.
\end{equation}

Next, let us pass to the limit in $I_3^\eps$ as $\eps \to 0$ under \eqref{Ncond}. Recall that 
$$
I_3^\eps = \int_{\Omega} \varphi(x) \, \chi_\eps(x) \int_{0}^t e^{-h_\eps(x) (t-s)} \int_{\R^N} J(x-y) \, \tilde u^\eps(y,s) \, dy ds dx.
$$ 
In order to do that, let us consider 
$$
\mathcal{S}_\eps(x,t) = \int_{0}^t e^{-h_\eps(x)(t-s)} \int_{\R^N} J(x-y) \tilde u^\eps(y,s)  dy ds
$$
defined for any $(x,t) \in \R \times \Omega$.
Since the sequences $\tilde u^\eps$ and $e^{-h_\eps(x)t}$ satisfy \eqref{weak} and \eqref{eq537} respectively, 
we get from Proposition \ref{propL} that 
$$
\mathcal{S}_\eps(x,t)  \to  
\mathcal{S}_0(x,t) = \int_{0}^t e^{-h_0(x)(t-s)} \int_{\R^N} J(x-y) \, u^*(y,s)  dy ds
$$
for any $(x,t) \in \R \times \Omega$. Furthermore, for all $t \in [0,T]$, we have from \eqref{Kest} that 
$$
|\mathcal{S}_\eps(x,t)| \leq \int_0^t \| J(x-\cdot) \|_{L^2(\Omega)}  \| u^\eps(s, \cdot) \|_{L^2(\Omega)} \, ds 
\leq T \, K \, \|J\|_{L^\infty(\R^N)} |\Omega|^{1/2}.
$$

Thus, it follows from Convergence Dominated Theorem that 
\begin{equation} \label{weakU}
\mathcal{S}_\eps(\cdot,t) \rightharpoonup \mathcal{S}_0(\cdot,t) \quad \textrm{ weakly in } L^2(\Omega)
\end{equation}
for each $t \in [0,T]$. 
In fact, we have that  
\begin{equation} \label{strongU}
\mathcal{S}_\eps(\cdot,t) \to \mathcal{S}_0(\cdot,t) \quad \textrm{ strongly in } L^2(\Omega)
\end{equation}
since 
$$
|\mathcal{S}_\eps(x,t)|^2 \leq  T^2 K^2 |\Omega| \|J\|_{L^\infty(\R^N)}^2,
$$
and then, due to Convergence Dominated Theorem again, we have 
\begin{equation} \label{690}
\| \mathcal{S}_\eps(\cdot,t) \|_{L^2(\Omega)}  \to  \| \mathcal{S}_0(\cdot,t) \|_{L^2(\Omega)}
\end{equation}
for all $t \in [0,T]$.
The strong convergence \eqref{strongU} follows from \eqref{weakU} and \eqref{690}
since we are working in the Hilbert space $L^2(\Omega)$.

Therefore, we can compute $I_3^\eps$ for each $\varphi \in L^2(\Omega)$.
From \eqref{strongU} we have
\begin{eqnarray} \label{691}
I_3^\eps & = & \int_\Omega \varphi(x) \, \chi_\eps(x) \,  \mathcal{S}_\eps(x,t) \, dx \nonumber \\
& \to & \int_\Omega \varphi(x) \, \mathcal{X}(x) \, \mathcal{S}_0(x,t) \, dx \nonumber \\
& = & \int_\Omega \varphi(x) \, \mathcal{X}(x) \int_0^t e^{-h_0(x)(t-s)} \int_{\R^N} J(x-y) \, u^*(y,s) \, dy dx ds.
\end{eqnarray}

Finally, let us pass to the limit in 
$$
I_2^\eps = \int_{\Omega} \varphi(x) \, \chi_\eps(x) \int_{0}^t e^{-h_\eps(x) (t-s)} \, f(x,\tilde u^\eps(x,s)) \, ds  dx.
$$
We first note that there exists $D>0$ such that
\begin{equation} \label{fB}
\sup_{s \in [0,T]} \| f(\cdot, \tilde u^\eps(\cdot,s) ) \|_{L^2(\Omega^\eps)} \leq D.
\end{equation}
In fact, since $f = g \circ m_{\Omega^\eps}$ with $m_{\Omega^\eps}(x,0)=0$, we have
\begin{eqnarray*}
\| f(\cdot, \tilde u^\eps(\cdot,s) ) \|_{L^2(\Omega^\eps)}^2 
& \leq & 2 \left[  \| f(\cdot, \tilde u^\eps(\cdot, s) ) - f(\cdot, 0 ) \|_{L^2(\Omega^\eps)}^2 + \| f(\cdot, 0 ) \|_{L^2(\Omega^\eps)}^2\right] \\
& \leq & 2 \left[  \int_\Omega L_g^2 | m_{\Omega^\eps}(x, \tilde u^\eps(x,s) ) |^2 dx + \int_\Omega g(0)^2 dx \right]
\end{eqnarray*}
where $L_g$ is the Lipschitz constant of the function $g$.

On the other hand, we get from Hardy-Littewood maximal inequality a constant $\hat C>0$ such that 
$$
\frac{1}{|B_\delta(x)|} \left| \int_{B_\delta(x)} \tilde u^\eps(y,s) \, dy \right| \leq \hat C \| \tilde u^\eps(\cdot,t) \|_{L^2(\Omega)} \quad \textrm{ a.e. } \Omega.
$$
Hence, due to \eqref{177} and \eqref{Kest}, there exists a constant $\tilde C>0$ such that  
\begin{eqnarray*}
| m(x, \tilde u^\eps(x,s) )| 
& = & \frac{|B_\delta(x)|}{|B_\delta(x) \cap \Omega^\eps|} \left| \frac{1}{|B_\delta(x)|} \int_{B_\delta(x)} \tilde u^\eps(y,s) \, dy \right| \leq \tilde C
\end{eqnarray*}
for all $\epsilon \in (0, \eps_0)$, $s \in [0,T]$ and $x \in \Omega$ proving \eqref{fB}. 

Moreover, we have that
\begin{equation} \label{SmX}
m_{\Omega^\eps}(\cdot, \tilde u^\eps(\cdot, s)) \to m_\mathcal{X} (\cdot, u^*(\cdot,s)) \quad \textrm{ strongly in } L^2(\Omega) \textrm{ as } \eps \to 0
\end{equation}
for all $s \in [0,T]$.
In fact, for each $x \in \R^N$ and $0< \eps < \eps_0$, we get from \eqref{177} and \eqref{weak} that 
\begin{eqnarray*}
m_{\Omega^\eps}(x, \tilde u^\eps(x,s)) & = & \left( \int_{B_\delta(x)} \chi^\eps(y) \, dy \right)^{-1} \left( \int_{B_\delta(x)} \tilde u^\eps(y,s) \, dy \right) \\
& \to & \left( \int_{B_\delta(x)} \mathcal{X}(y) \, dy \right)^{-1} \left( \int_{B_\delta(x)} u^*(y,s) \, dy \right) 
= m_\mathcal{X}(x, u^*(x,s))
\end{eqnarray*}
where $m_\mathcal{X}$ is defined in \eqref{mX}.
Hence, since $m_{\Omega^\eps}(\cdot, \tilde u^\eps)$ is uniformly bounded in $\Omega \times [0,T]$, 
we can argue as in \eqref{strongU} to obtain \eqref{SmX} by Convergence Dominated Theorem. 

Now, using that $g$ is a Lipschitz continuous function, we can get from \eqref{fB} and \eqref{SmX} that 
\begin{equation} \label{fweak}
f(\cdot, u^\eps) \rightharpoonup f_\mathcal{X} (\cdot, u^*) \textrm{ weakly$^*$ in } L^\infty([a,b] ; L^2(\Omega))
\end{equation}
where $f_\mathcal{X}$ is defined in \eqref{fX}.
Consequently, we can argue as in \eqref{strongU} again, to get that
$$
\int_{0}^t e^{-h_\eps(x) (t-s)} \, f(x,\tilde u^\eps(x,s)) \, ds \to 
\int_{0}^t e^{-h_0(x) (t-s)} \, f_\mathcal{X}(x, u^*(x,s)) \, ds, \quad \textrm{ as } \eps \to 0.
$$

Consequently, we can pass to the limit in $I_2^\eps$ getting 
$$
I_2^\eps \to \int_{\Omega} \varphi(x) \, \mathcal{X}(x) \int_{0}^t e^{-h_0(x) (t-s)} \, f_\mathcal{X}(x, u^*(x,s)) \, ds  dx, \quad \textrm{ as } \eps \to 0.
$$

Thus, the limit of the integral equation \eqref{550} is 
\begin{equation} \label{700}
\begin{array}{l}
\displaystyle \int_{\Omega} \varphi(x) \, u^*(x,t) \, dx = 
\int_{\Omega} \varphi(x) \mathcal{X}(x) \left[ e^{-h_0(x)t} \, u_0(x) + \int_{0}^t e^{-h_0(x)(t-s)} \, f_\mathcal{X}(x, u^*(x,s)) \, ds \right] dx  \\[10pt] 
\qquad \quad \quad \displaystyle + 
 \int_{\Omega} \varphi(x) \, \mathcal{X}(x) \int_{0}^t e^{-h_0(x)(t-s)} \int_{\R^N} J(x-y) \, u^*(y,s) \, dy ds dx, \quad \forall \varphi \in L^2(\Omega),
\end{array}
\end{equation}
which implies 
\begin{equation} \label{exu*}
\begin{array}{l}
\displaystyle u^*(x,t) = e^{-h_0(x)t} \, \mathcal{X}(x) \, u_0(x) + \int_0^t e^{-h_0(x)(t-s)} \, \mathcal{X}(x) \, f_\mathcal{X}(x, u^*(x,s)) \, ds \\[10pt] 
\qquad \qquad \qquad \qquad \displaystyle + \int_0^t e^{-h_0(x)(t-s)} \, \mathcal{X}(x) \int_{\R^N} J(x-y) \, u^*(y,s) \, dy ds 
\end{array}
\end{equation}
for all $t \in [0,T]$ and a.e. $x$ in $\Omega$.
Thus, $u^* \in C^1([0,T];L^2(\Omega))$ and satisfies  
\begin{equation}
\begin{gathered}
u^*_t(x,t) = \mathcal{X}(x) \int_{\R^N} J(x-y) \, u^* (y,t) \, dy  - h_0(x) \, u^*(x,t) + \mathcal{X}(x) \, f_\mathcal{X}(x, u^*(x,t)),   \\
u^* (x,t) \equiv 0 \quad \textrm{ for } x \in \R^N \setminus \Omega \\
u^*(0,x) = \mathcal{X}(x) \, u_0(x) 
\end{gathered} 
\end{equation}
which can be rewritten as \eqref{715} under assumption \eqref{Ncond} 
with $\Lambda \in L^\infty(\R^N)$ given by
$$
\Lambda(x) = h_0(x) - \mathcal{X}(x).
$$

Finally, let us notice that $u^*$ is unique from Remark \ref{reE}.
Indeed, if we re-scale the time $t$ with $t = \mathcal{X}^{-1}(x) \, \tau$ and set 
$$
w(x,\tau) = u^*(x, {\mathcal{X}(x)}^{-1} \tau)
$$
we have that $w$ satisfies equation \eqref{reESolution} for $h(x) = {\mathcal{X}(x)}^{-1} \Lambda(x) \in L^\infty(\Omega)$. 
Thus, $u^*$ is unique which implies that the sequence $u^\eps$ converges weakly to $u^*$ as $\eps \to 0$. 
In this way, we conclude the proofs of Theorems \ref{theoD} and \ref{theoN}. 
\end{proof}

\section{A nonlocal equation with local nonlinearity}  \label{deltaS}

Now, we obtain a nonlocal equation with local nonlinearity from the limit problem given by Theorems \ref{theoD} and \ref{theoN}. 
We consider the limit problem depending on the parameter $\delta$, that is, the  equation associated to $u^{\delta} \in C^1([0,T];L^2(\Omega))$ which satisfies  
\begin{equation}\label{prob_delta}
\begin{gathered}
u^\delta_t(x,t) = \mathcal{X}(x) \int_{\R^N} J(x-y) \, u^\delta (y,t) \, dy  - h_0(x) \, u^\delta(x,t) + \mathcal{X}(x) \, f_\mathcal{X}(x, u^\delta(x,t)),   \\
u^\delta (x,t) \equiv 0 \quad \textrm{ for } x \in \R^N \setminus \Omega \\
u^\delta(0,x) = \mathcal{X}(x) \, u_0(x). 
\end{gathered} 
\end{equation}

The existence of the family of solutions $u^\delta$ of \eqref{exist_delta} are guaranteed by Proposition \ref{exist_delta}. Also, there exists $C>0$, 
independent of $\delta>0$, such that, for any bounded interval $[a,b] \subset \R$, 
\begin{equation} \label{Kest0}
\sup_{t \in [a,b]} \| u^\delta(\cdot,t) \|_{L^2(\Omega)} \leq C.
\end{equation}
Hence, if $\tilde \cdot$ denotes the extension by zero to the whole space $\R^N$, we also get that 
\begin{equation} \label{estdel}
\sup_{t \in [a,b]} \| \tilde u^\delta(\cdot,t) \|_{L^2(\Omega)} \leq C,
\end{equation}
and then, $\tilde u^\delta$ sets a uniformly bounded family in $L^\infty \left( [a,b] ; L^2(\Omega)  \right)$. 

Since $L^1 \left( [a,b] ; L^2(\Omega)  \right)$ is separable, 
we can extract a subsequence still set by $\tilde u^\delta$ such that
\begin{equation} \label{weak_delta}
\tilde u^\delta \rightharpoonup \bar{u} \textrm{ weakly$^*$ in } L^\infty([a,b] ; L^2(\Omega)),
\end{equation}
for some $\bar{u} \in L^\infty \left( [a,b] ; L^2(\Omega)  \right)$. Notice that $\bar{u}(x,t) \equiv 0$ in $\R^N \setminus \Omega$.

Then, we can proceed as in Section \ref{limit} to prove Theorem \ref{theo_delta}. Since the proof is very similar, we will leave the details to the reader.
Here we just pass to the limit in the nonlinear term
$$
\begin{gathered}
\int_\Omega \varphi(x) \mathcal{X}(x) f_{\mathcal{X}}(x, u^\delta(\cdot,t)) \, dx = 
\int_\Omega \varphi(x) \mathcal{X}(x) \, g\left( \frac{1}{\int_{B_\delta(x)} \mathcal{X}(y) dy} \int_{B_\delta(x)} u^\delta(y,t) dy \right) dx \\ 
\qquad = \int_\Omega \varphi(x) \mathcal{X}(x) \, g\left( \frac{|B_\delta(x)|}{\int_{B_\delta(x)} \mathcal{X}(y) dy} \frac{1}{|B_\delta(x)|} \int_{B_\delta(x)} u^\delta(y,t) dy \right) dx. 
\end{gathered}
$$
But, it is a direct consequence of Lebesgue Differentiation Theorem and the uniform estimate given by Corollary \ref{cordelta}.
Indeed, from \eqref{estH1} and the compact embedding from $H^1$ into $L^2$, we have $u^\delta(\cdot,t) \to \bar u(\cdot,t)$ strongly in $L^2(\Omega)$, as $\delta \to 0$, for any $t \in \R$. Thus, from Lebesgue Differentiation Theorem, we obtain
$$
\int_\Omega \varphi(x) \mathcal{X}(x) f_{\mathcal{X}}(x, u^\delta(\cdot,t)) \, dx \to \int_\Omega \varphi(x) \mathcal{X}(x) g \left( {\mathcal{X}}^{-1}(x) \,  \bar u(x,t) \right) \, dx
$$
which leads us to the limit equation \eqref{1067}.

\section{Appendix}

In this section we just compute the derivatives of the map $\Phi: \R^N \mapsto \R$ given by
$$
\Phi(x) = \int_{B(x)} u(y) \, dy
$$
where 
$
B(x) = \{ y \in \R^N \, : \, | x- y | < R \} 
$
is a ball of radius $R>0$ and $u$ is a smooth function.
Notice that, due to \cite[Lebesgue-Radon-Nikodym Theorem]{Rudin}, $\Phi$ is absolutely continuous whenever $u \in L^1(\R^N)$. 

From Taylor's formula, we have for any $v \in \R^N$ and $t \in \R$ that
$$
\Phi(x+tv) - \Phi(x) = \int_{B(x)} \{ u(y+tv) - u(y) \} dy = t \int_{B(x)} \nabla u(y) \cdot v dy + O(t^2). 
$$
Thus, by Green's Identity, we get the following expression 
\begin{equation} \label{derivative}
\nabla \Phi(x) \cdot v = \int_{B(x)} \nabla u(y) \cdot v dy = \int_{\partial B(x)} u(y) \, v \cdot N \, dS
\end{equation}
where $N$ is the normal vector on the boundary of the ball $\partial B(x)$. 
Finally, we observe that this formula \eqref{derivative} is in agreement with \cite[Theorem 1.11]{Henry}.

\vspace{0.7 cm}

%{\bf Acknowledgements.} 

\bibliography{bib_sastregom}

\typeout{get arXiv to do 4 passes: Label(s) may have changed. Rerun}

\end{document}